\documentclass[12pt, a4paper]{amsart} 
\usepackage[T2A]{fontenc}
\usepackage[utf8]{inputenc}
\usepackage[dvips]{graphicx}
\usepackage{geometry}
\usepackage{amsmath, amssymb, amsfonts, amsthm}

\usepackage{graphicx}
\usepackage{tikz}
\usepackage{hyperref}
\usepackage{bbm}

\usepackage{comment}
\usepackage[utf8]{inputenc}
\usepackage[T2A]{fontenc}
\usepackage[english]{babel}

\newtheorem*{ack}{Acknowledgements}
\newtheorem{theorem}{Theorem}
\newenvironment{proof}{\noindent{\bf Proof:}}{$\hfill \Box$ \vspace{10pt}}

\newtheorem{lem}{Lemma}

\newtheorem{defin}{Definition}
\newtheorem{theom}{Theorem}
\newtheorem{rem}{Remark}
\newcommand{\hel} {
\hskip2.5pt{\vrule height7pt width.5pt depth0pt}
\hskip-.2pt\vbox{\hrule height.5pt width7pt depth0pt}
\, }
\newcommand{\restr}{\hel}

\begin{document}

\title[On the local H\"older boundary smoothness of an analytic function\ldots]{On the local H\"older boundary smoothness of an analytic function in the unit ball compared with the smoothness of its modulus.}
\author{Ioann Vasilyev}
\thanks{This research was Supported by RFBR, grant no. 18-31-00250.}
\address{Universit\'e Paris-Est, LAMA (UMR 8050), 5 Boulevard Descartes, 77454, Champs sur Marne, France}
\address{St.-Petersburg Department of V.A. Steklov Mathematical Institute, Russian Academy of Sciences (PDMI RAS), Fontanka 27, St.-Petersburg, 191023, Russia}
\email{milavas@mail.ru}

\date{\today}

\begin{abstract}
Local boundary smoothness of an analytic function $f$ in the unit ball of $\mathbb C^n$ is compared to the smoothness of its modulus. We prove that in dimensions $2$ and higher two different (and natural) conditions imposed on the zeros of $f$ imply two different drops of its smoothness compared to the smoothness of $|f|$. We also show that some of the drops are the best possible. 
\end{abstract}

\maketitle
\section{Introduction}
It is fairly well known that if $f$ is an analytic function in the unit disc $\mathbb D$ continuous up to the boundary, then the H\"older continuity of $f$ is less in general than that of $\phi=|f|\restr_{\partial\mathbb D}$. To discuss this phenomenon in more detail, recall the classical inner--outer factorization. Every (say, bounded) analytic function $F$ in $\mathbb D$ can be represented in the form $F=BSG,$ where $B$ is a Blaschke product, $S$ is a zero--free bounded analytic function whose boundary values are of modulus $1$ a.e. on $\mathbb T=\partial \mathbb D,$ and $G$ is an outer function. This means that $G$ is represented in the form
\begin{equation}
\label{tratata}
G(z)=\mathcal{O}_\varphi(z)=\exp\biggl[\frac{1}{2\pi}\int\limits_0^{2\pi}\frac{e^{i\theta}+z}{e^{i\theta}-z}\log\varphi(e^{i\theta})d\theta\biggl],
\end{equation}
\noindent where $\phi$ is a positive function satisfying $\log\phi \in L^p(\mathbb T).$ Note that $|f|= \phi$ a.e. on $\mathbb T$. We refer the reader to the books~\cite{nikol} and~\cite{garnet} for the properties of the outer functions on the unit disc.

It turns out that the presence of the Blaschke product $B$ hampers drastically the situation. By way of example we may consider the function $z^n,$ which is H\"older continuous, but its modulus of continuity is far as nice (especially if $n$ is large) as that of $|z^n \restr_{\mathbb T}|\equiv 1$. So, we restrict ourselves to the case where $B$ is absent, i.e., $f$ has no zeros in $\mathbb D$. It should be noted that this case is not so far of that of an outer $f$ because then $f(rz)$ is outer for every $r<1$ (below we shall see how this observation is used in the case of the complex ball). So, we state a theorem for an outer $f$, see~\cite{hsh}.
\begin{theom}(Carleson--Jacobs--Havin--Shamoyan)
\label{cjhsh}
Let $\alpha\in (0,1).$ Suppose that $f: \mathbb D \rightarrow \mathbb C$
is an outer function which has an $\alpha$-H\"older modulus on the boundary circle $\mathbb T$ of the open unit disc $\mathbb D$. Then the function $f$ itself is $\alpha/2$-H\"older on $\mathbb T$.
\end{theom}

Several remarks are in order. First, Theorem~\ref{cjhsh} holds for all indices $\alpha \in \mathbb R_+$. Reportedly, this result was first proved by L. Carleson. Nevertheless, the only published proof of this theorem is published in the book~\cite{shir1}. We refer the reader to the paper~\cite{vkm} for a more detailed discussion of the history of this theorem. Second, we would like to draw the reader's attention to the fact that Theorem~\ref{cjhsh} was used by J. Brennann in his paper~\cite{bren}, where with the help of this result he characterized planar domains on which any analytic function admits polynomial approximation in the $L^p$ metric. Another application of the global Carleson--Jacobs--Havin--Shamoyan theorem was found in the paper~\cite{abak}, where the authors use it in order to classify cyclic subspaces of the harmonic Dirichlet spaces. We also mention the paper~\cite{mash} by Mashreghi and Shabankhah, where the Carleson--Jacobs--Havin--Shamoyan theorem was used to compare zero sets and uniqueness sets of functions in Dirichlet spaces. One more remark on Theorem~\ref{cjhsh} is that it was cited in papers~\cite{bruort},~\cite{dyak1},~\cite{dyak2} and~\cite{taywil}.

Surprisingly, since 2012 the interest to this range of problems arose again. Specifically, the following natural question was raised: suppose that $\phi=|f|$ satisfies a H\"older condition at one point of $\mathbb T$ only. What can be said about $f$ at the same point? The following local version of Theorem~\ref{cjhsh} was proved in the paper~\cite{vkm}.

\begin{theom}(Kislyakov--Vasin--Medvedev)
\label{local}
Let $\alpha\in (0,2).$ Suppose that $f: \mathbb D \rightarrow \mathbb C$
is an outer function  which has an $\alpha$-H\"older modulus at some point $\xi\in\mathbb T.$ Then for all arcs $I\subseteq \mathbb T$ containing $\xi$ the mean oscillation $\nu(f,I)$ satisfies
\begin{equation}
\label{pampum}
\nu(f,I):=\inf\limits_{a\in \mathbb C} \frac{1}{|I|}\int\limits_{I}\left|f(z)-a\right|dz\leq C |I|^{\frac{\alpha}{2}},
\end{equation}
where $C$ depends on $\alpha, \|\log |f|\|_{L^1(\mathbb T)}$ and the H\"older norm of $|f|$ only. 
\end{theom}
Properties of the mean oscillation $\nu$ and its connections with the local and global H\"older and Lipschitz smoothness classes are discussed in the papers~\cite{dev} and~\cite{vkm}.

If one looks at the proofs of the above results in~\cite{hsh} and~\cite{vkm}, it becomes clear that an obstruction for an uncontrollable smoothness drop of $f$ compared to $\phi=|f|$ is in the integrability of $\log \phi$ on $\mathbb T,$ which is true automatically. It has turned out that a stronger condition on $\log \phi$ implies a smaller drop.
\begin{theom}(Kislyakov--Vasin--Medvedev)
\label{kvm}
Let $\alpha\in (0,2).$ Suppose that $f: \mathbb D \rightarrow \mathbb C$ is an analytic function in the Smirnov class without zeros inside $\mathbb D$ which has an $\alpha$-H\"older modulus at a point $\xi\in\mathbb T.$ Suppose that $B_p:=\int_{\mathbb T} |\log |f||^p< \infty$ for some $p>1$. Then for all arcs $I\subseteq \mathbb T$ containing $\xi$ the mean oscillation $\nu(f,I)$ satisfies
$$\nu(f,I)\leq C |I|^{\frac{\alpha}{2-\frac{1}{q}}},$$
where $q$ is the H\"older conjugate of $p$, i.e. $1/p+1/q=1$ and $C$ depends on $B_p, \alpha$ and the H\"older norm of $|f|$ only. 
\end{theom}
\noindent See~\cite{shir2} for the global setting and~\cite{vkm} for the local one.

Note that Theorems~\ref{local} and~\ref{kvm} provide a significant improvement of Theorem~\ref{cjhsh} and moreover these local estimates imply the global ones with the genuine (and not the mean integral) H\"older regularity. To illustrate this, we mention the following fact, whose proof can be found, e.g., in~\cite{dev}: if there is a uniform bound of the mean oscillation of a function $f$ on some interval, then $f$ is H\"older on this interval.

The next step was the passage to higher dimensions. In~\cite{shir}, Shirokov proved that a zero--free function analytic in the unit ball $\mathbb B^n$ of $\mathbb C^n$ and continuous up to the boundary must be $\alpha/2$-H\"older if the modulus of its restriction to $\mathbb B^n$ is $\alpha$-H\"older.

The present paper is devoted to a local version of this result. In the course of the study of this matter, a phenomenon was discovered that was invisible (and cannot occur) in dimension $1$. Specifically, in Theorem~\ref{kvm} above the $1/2$ smoothness drop occurs without any assumptions on the boundedness of $\phi$ far from the point where it is H\"older continuous: it only suffices that $\log \phi$ is integrable and is defined by the formula~\eqref{tratata}.

Outer functions exist also in the ball; they are defined as follows.
\begin{defin}
\label{uneshnyaya}
\textup{A function} $f: \mathbb B^n\rightarrow \mathbb C$ \textup{is called}\textbf{ outer}, \textup{if for all} $z\in \mathbb B^n$ \textup{one has}
$$f(z)=\exp\biggl[\int\limits_{\mathbb S^n}\left(2C(z,\xi)-1\right)\mathrm{Re}(\log f(\xi))d\sigma(\xi)\biggr],$$
\textup{where $\sigma$ denotes the standard rotation-invariant Borel probability measure on} $\mathbb S^n.$
\end{defin}
\begin{rem}
We remind to the reader that the Cauchy kernel $C(z,\xi)$ in the unit ball is defined as $C(z,\xi)=(1-\langle z,\xi\rangle)^{-n}.$
The function $\int_{\mathbb S^n} C(z,\xi)f(\xi)d\sigma(\xi),$ will be sometimes referred to as the ``convolution'' of $f$ with the Cauchy kernel.
\end{rem}

In order to state the main results of the article, we recall three more definitions.
\begin{defin}
\textbf{The nonisotropic quasimetric} \textup{on the $n$-dimensional unit sphere $\mathbb S^n$ is defined as follows: for $u,v\in \mathbb S^n$,}
$d(u,v):=|1-\langle u,v\rangle|$. \textbf{Nonisotropic ball} \textup{is a set of the form} $Q=\{z\in \mathbb S^n: d(z,\xi)\leq r\}$ \textup{with some} $\xi\in \mathbb S^n$ \textup{and} $r\geq 0.$
\end{defin}
In analogy with~\eqref{pampum} we define a multidimensional mean oscillation measuring smoothness.
\begin{defin}
\textup{For a locally summable function $f:\mathbb S^n\rightarrow \mathbb C$ and a ball $Q\subseteq \mathbb S^n$ the} \textbf{ mean oscillation measuring smoothness $\nu(f,Q)$} \textup{is defined as follows}
$$\nu(f,Q):=\inf\limits_{a\in \mathbb C} \frac{1}{|Q|}\int\limits_{Q}\left|f(z)-a\right|d\sigma(z).$$
\end{defin}
The following definition will be very important for us, especially in the proof of Theorem~\ref{grglthm2}.
\begin{defin}
\textup{Let $f$ be an analytic in $\mathbb B^n$ function continuous up to the boundary. Then it satisfies the following inequality}
\begin{equation}
\label{eq-10}
\sup\limits_{\xi\in \mathbb S^n}\int\limits_{\mathbb T}|\log|f(\xi\lambda)||d\lambda=: B_0<\infty,
\end{equation}
\textup{consult~\cite{rud} or~\cite{shir} for the proof. In this case we shall say that $f$ satisfies the} \textbf{``slice'' condition}. 
\end{defin}
\begin{rem}
We fix the following notations once and for all: $\mathbbm{1}:=(1,0,\ldots,0);$ for a nonisotropic ball $Q\subseteq \mathbb S^n$ its radius will be denoted by $l(Q).$
\end{rem}

It turns out that if we impose additionally only a H\"older condition on $\phi$ (say) at the point $\mathbbm{1}$, we do not obtain a $1/2$-drop. Moreover, to say at least something, we should suppose that $\log\phi\in L^p(\mathbb S^n)$ with some $p>1.$ The following result is sharp, as we shall see in the third section of this paper (see Theorem~\ref{grglthm5}).
\begin{theorem}
\label{grglthm1}
Let $\alpha\in (0,1).$ Suppose that $f: \mathbb B^n\rightarrow \mathbb C$ is an outer function such that for all $t\in \mathbb S^n$ one has $|\phi(t)-\phi(\mathbbm{1})|\leq C_0d(t,\mathbbm{1})^{\alpha}$, where $\phi:=|f|.$ Suppose also that $B_p:=\int_{\mathbb S^n} |\log \phi|^p< \infty$ for some $p>1$. Then for all nonisotropic balls $Q\subseteq \mathbb S^n$ containing the point $\mathbbm{1}$ the mean oscillation $\nu(f,Q)$ satisfies
$$\nu(f,Q)\leq C l(Q)^{\frac{\alpha}{n+1-\frac{n}{q}}},$$
where $q$ is the H\"older conjugate of $p$ and $C$ depends on $B_p, C_0, \alpha$ and $n$ only. 
\end{theorem}
\begin{rem}
Note that $\alpha/(n+1-n/q)=\alpha p/(p+n).$
\end{rem}

However, if we suppose that $f$ is continuous up to the boundary in $\mathbb B^n$ and zero--free (recall that then $f(rz)$ is outer), then we regain the $1/2$-drop.
\begin{theorem}
\label{grglthm2}
Let $\alpha\in (0,1).$ Suppose that $f: \mathbb B^n\rightarrow \mathbb C$ is an analytic function without zeros inside $\mathbb B^n$, continuous up to the boundary $n$-dimensional unit sphere $\mathbb S^n$ such that for all $t\in \mathbb S^n$ one has $|\phi(t)-\phi(\mathbbm{1})|\leq C_0d(t,\mathbbm{1})^{\alpha}$. Then for all nonisotropic balls $Q$ containing the point $\mathbbm{1}$ the mean oscillation $\nu(f,Q)$ satisfies
$$\nu(f,Q)\leq C l(Q)^{\frac{\alpha}{2}},$$
where $C$ depends only on $B_0$(see~\eqref{eq-10}), $C_0, \alpha$ and $n$. 
\end{theorem}
\begin{rem}
If an estimate of the type $\nu(f,Q)\leq Cl(Q)^{\beta}$ holds for all balls $Q$ containing some point $\xi\in \mathbb S^n$ with some $\beta>0$ and $C$ independent of $Q$ we shall sometimes say that $f$ is $\beta$-H\"older ``in average'' at $\xi$.  
\end{rem}

While proving Theorems~\ref{grglthm1} and~\ref{grglthm2} the author has been inspired by an approach developed by Kislyakov and coauthors (see~\cite{vkm}). Indeed, the proofs of Theorems~\ref{grglthm1} and~\ref{grglthm2} resemble the $1$-dimensional pattern of Theorem~\ref{kvm}, but only up to a certain point. As it was in dimension $1$, an obstruction for an uncontrollable smoothness drop is the integrability of (some power of) $\log \phi$. The difference is that in Theorem~\ref{grglthm1} this integrability is against the surface measure on $\mathbb S^n$ and in Theorem~\ref{grglthm2} we have the same on every one-dimensional slice. The latter feature leads to new calculations at the core of the proof of Theorem~\ref{grglthm2}. On top of that, in Theorem~\ref{grglthm2} we are dealing with zero--free analytic (and not outer) functions, which makes, as we shall see, the proof of this theorem more complicated than the proofs of Theorems~\ref{grglthm1} and~\ref{kvm}.

As it seems to the author, it is plausible that there are versions of Theorems~\ref{grglthm1} and~\ref{grglthm2} that hold true in a more general setting, namely in the context of the holomorphic functions defined on more general domains in $\mathbb C^n$. The author does not know whether the theorems proved here hold if one considers $\alpha$ strictly bigger than one in those. Neither does he know if the strong $\alpha$-H\"older condition can be substituted with a weaker ``average'' one. The author plans to prove these generalizations in the nearest future.
\begin{ack}
The author is kindly grateful to his scientific adviser academician Sergei V. Kislyakov for having posed the problem, for a number of helpful suggestions and for help in preparation of this article.
\end{ack}

\section{Proof of Theorem~\ref{grglthm2}}
We are acting in the following way. We start with a technical result, which we, nevertheless, call a Theorem by the reason of some nontrivial (at least in our opinion) estimates included in its proof. With help of this theorem we shall later obtain the desired bound on the mean oscillation $\nu(f,Q).$
\begin{theorem}
\label{sovsemustal}
The following estimates hold under the conditions of Theorem~\ref{grglthm2}.
\begin{enumerate}
\item For all nonisotropic balls $Q\subseteq \mathbb S^n$ containing the point $\mathbbm{1}$
\begin{equation*}
\nu(f,Q)\leq C l(Q)^{\alpha}+ \phi(\mathbbm{1}).
\end{equation*}
\item If $\phi(\mathbbm{1})>0,$ then for all nonisotropic balls $Q\subseteq \mathbb S^n$ containing the point $\mathbbm{1}$ and such that $l(Q)\leq \left(\phi(\mathbbm{1})/2C_0\right)^{1/\alpha}$ holds
\begin{equation*}
\nu(f,Q)\leq C l(Q)^{\alpha}+C\frac{l(Q)^2}{\phi(\mathbbm{1})^{\frac{2}{\alpha}-1}},
\end{equation*}
\end{enumerate}
where the constant $C$ depends only on $n, C_0$ and $B_0.$
\end{theorem}
\begin{rem}
From now on the sign $\lesssim$ indicates that the left-hand part of an inequality is less than the right-hand part multiplied by a constant a $C$ depending only on $n, C_0$ and $B_0.$
\end{rem}
\begin{proof}
We can suppose (in the both theorems) that $\phi(\mathbbm{1})\leq 1$. We argue by contradiction. Indeed, assume that $\phi(\mathbbm{1})>1$. Since the function $\phi$ is H\"older at the point $\mathbbm{1}$, we have $\phi(\mathbbm{1})\leq 2\max(\phi(\xi),C_0d(\xi,\mathbbm{1})^{\alpha})$ for all $\xi \in \mathbb S^n$. As a consequence, we infer the inequality
\begin{equation*}
0\leq\log\phi(\mathbbm{1})\leq\log 2 + |\log \phi(\xi)|+|\log C_0|+\alpha |\log d(\xi,\mathbbm{1})|.
\end{equation*}
Integration of the last line yields the inequality $|\log\phi(\mathbbm{1})|\lesssim |\log C_0|+B_0+\kappa,$ with a constant $\kappa$ depending on $\alpha$ and $n$ only. From here we readily deduce that $\phi(\mathbbm{1})\lesssim \exp (|\log C_0|+B_0+\kappa)$. That is why we can consider from the very beginning the function $\widetilde{f}(z)=f(z)/\phi(\mathbbm{1})$ instead of $f$. Indeed, this makes sense since the function $\widetilde{f}$ (which is obviously continuous up to $\mathbb S^n$) is zero-free, satisfies the  ``slice'' condition~\eqref{eq-10} and moreover the corresponding value of the supremum there is controlled by the constants $B_0$ and $C_0$. We shall further write $f$ instead of $\widetilde{f}$.

With no loss of generality, we suppose that $f(0)$ is a real number (because the general case follows from the observation that the function $g(z)=f(z)\cdot\overline{f(0)}/|f(0)|$ satisfies $g(0)\in \mathbb R$). Since $f$ is an analytic function without zeros, we are allowed to write the following representation for the functions $f_r(\xi):=f(r\xi), r<1$:
\begin{multline*}
f_r(z)=\exp\biggl[\int\limits_{\mathbb S^n}\left(2C(z,\xi)-1\right)\text{Re}(\log f_r(\xi))d\sigma(\xi)\biggr]=\\
\exp\biggl[\int\limits_{\mathbb S^n}\left(2C(z,\xi)-1\right)\log |f_r(\xi)|d\sigma(\xi)\biggr],
\end{multline*}
where $C$ is the Cauchy kernel for the unit sphere in the $n$-dimensional complex space, contact~\cite{rud} for the proof. Hence
\begin{multline*}
f_r(z)=\exp\biggl[\int\limits_{\mathbb S^n}\left(\text{Re}(2C(z,\xi)-1)+i\text{Im}(2C(z,\xi)-1)\right)\log|f_r(\xi)|d\sigma(\xi)\biggr]= \\
\phi_r(z)\exp\biggl[i\int\limits_{\mathbb S^n}\text{Im}\left(2C(z,\xi)-1\right)\log |\phi_r(\xi)|d\sigma(\xi)\biggr]=:\phi_r(z)e^{iG(z)},
\end{multline*}
where we write $\phi_r(\xi):=\phi(r\xi)=|f(r\xi)|$ for sake of brevity. Next, we estimate $\nu(f,Q)$ for some fixed nonisotropic ball $Q$ such that $\mathbbm{1}\in Q\subseteq\mathbb S^n.$ First, we choose $a:=\phi(\mathbbm{1})e^{ic_0}$ for some positive constant $c_0.$ Note that since $f$ is continuous at any point $\xi$ of the boundary sphere, we infer that for any $\varepsilon>0$ there exists $\delta>0$ such that if $1-\delta<r<1$, then $|\phi_r(\xi)-\phi(\xi)|\leq \varepsilon$ and $|f_r(\xi)-f(\xi)|\leq \varepsilon$. From now on we consider only these $r$'s. From here we see that 
\begin{multline*}
\nu(f,Q)\leq \\
\frac{1}{|Q|}\int\limits_{Q}|f(z)-\phi(\mathbbm{1})e^{ic_0}|\leq\frac{1}{|Q|}\int\limits_{Q}|f(z)-f_r(z)|+ \frac{1}{|Q|}\int\limits_{Q}|\phi_r(z)e^{iG(z)}-\phi(\mathbbm{1})e^{ic_0}| \leq\\
 \varepsilon+ \frac{1}{|Q|}\int\limits_{Q}|(\phi_r(z)-\phi(\mathbbm 1))e^{iG(z)}|+ \frac{1}{|Q|}\int\limits_{Q}\phi(\mathbbm 1)|e^{iG(z)}-e^{ic_0}| \leq \varepsilon + C l(Q)^{\alpha}+2\phi(\mathbbm 1),
\end{multline*}
so the first claim of Theorem~\ref{sovsemustal} follows.

In order to prove the second part of Theorem~\ref{sovsemustal}, we choose $$c_0:=\int\limits_{\mathbb S^n \backslash 2Q} \text{Im}\left(2C(\mathbbm{1},\xi)-1\right)\cdot\left(\log\phi_r(\xi)-\log\phi(\mathbbm{1})\right) d\sigma(\xi).$$ Hence, thanks to the fact that the integral of the function $\log\phi(\mathbbm{1})\cdot\left(2C(\mathbbm{1},\xi)-1\right)$ over the unit sphere equals zero, we infer the inequality
\begin{equation*}
\nu(f,Q)\leq\frac{1}{|Q|}\int\limits_{Q}|\phi-\phi(\mathbbm{1})| + A+\varepsilon,
\end{equation*}
where $A$ by definition is equal to the following sum of integrals 
\begin{multline*}
C_1+D:= \frac{\phi(\mathbbm{1})}{|Q|}\biggl[\int\limits_{Q}\int\limits_{\mathbb S^n}\chi_{2Q}\cdot\text{Im}\left(2C(z,\xi)-1\right)\cdot\left(\log\phi_r(\xi)-\log\phi(\mathbbm{1})\right) d\sigma(\xi)d\sigma(z)\biggr. +\\
\biggl.\int\limits_{Q}\int\limits_{\mathbb S^n\backslash 2Q}\left(\text{Im}\left(2C(z,\xi)-1\right)-\text{Im}\left(2C(\mathbbm{1},\xi)-1\right)\right)\cdot\bigl(\log\phi_r(\xi)-\log\phi(\mathbbm{1})\bigr) d\sigma(\xi)d\sigma(z)\biggr].
\end{multline*}

We shall first estimate the integral $C_1.$ Note that for all $\xi\in Q$ and for all $\varepsilon$ small enough,
\begin{equation}
\label{strekozzza}
\phi_r(\xi)-\phi(\mathbbm 1)\leq C_0|\xi-\mathbbm 1|^{\alpha}+\varepsilon\leq C_0 l(Q)^{\alpha}+\varepsilon\leq \phi(\mathbbm 1),
\end{equation}
where the last inequality here follows from the conditions imposed on $Q$. Hence $\phi_r(\xi)\leq 2\phi(\mathbbm 1)$. Referring to this and to the trivial inequality $|\log\mu-\log\eta|\leq |\mu-\eta|/\min(\mu,\eta),$ which is valid for all $\mu,\eta>0,$ we infer the estimate
\begin{equation*}
|\log\phi_r(\xi)-\log\phi(\mathbbm{1})|\leq \frac{2}{\phi(\mathbbm{1})} |\phi_r(\xi)-\phi(\mathbbm{1})|.
\end{equation*}
This fact along with the $L^2$-boundedness of the singular integral represented by the convolution with the Cauchy kernel yields the usual trivial bound for $C_1:$
\begin{multline*}
C_1^2\leq \frac{ \phi(\mathbbm{1})^2}{|Q|}\int\limits_{Q}\Big(\int\limits_{\mathbb S^n}\chi_{2Q}\cdot\text{Im}\left(2C(z,\xi)-1\right)\cdot\left(\log\phi_r(\xi)-\log\phi(\mathbbm{1})\right) d\sigma(\xi)\Big)^2d\sigma(z)\lesssim\\
 \frac{\phi(\mathbbm{1})^2}{|Q|}\int\limits_{Q} |\log\phi_r(\xi)-\log\phi(\mathbbm{1})|^2 d\sigma(\xi)\lesssim\frac{1}{|Q|}\int\limits_{Q} |\phi_r(\xi)-\phi(\mathbbm{1})|^2 d\sigma(\xi)\lesssim l(Q)^{2\alpha}+\varepsilon,
\end{multline*} 
where the last inequality follows from the fact that the function $\phi$ is $\alpha$-H\"older at the point $\mathbbm 1$ and from the conditions imposed on $r$. 

Next, we estimate the term $D$. We introduce the following decomposition of the unit $n$-sphere: $\mathbb S^n= \bigcup_{j=1}^{m} \Omega_j,$ where $\Omega_j:=\{z\in\mathbb S^n: d(z,\mathbbm{1})\in(2^j l(Q),2^{j+1}l(Q))\}$ (here $m$ is the smallest natural number, such that $2^{m+1}Q\supseteq S$). We further use this decomposition in the estimate of the term $D$:
\begin{equation*}
D\lesssim\frac{\phi(\mathbbm{1})}{|Q|}\int\limits_{Q}\sum\limits_{j=1}^m\int\limits_{\Omega_j}|\log\phi_r(\xi)-\log\phi(\mathbbm{1})|\cdot|\widetilde{C}(z,\xi)-\widetilde{C}(\mathbbm{1},\xi)|d\sigma(\xi)d\sigma(z),
\end{equation*}
where $\widetilde{C}(z,\xi)$
is the imaginary part of the Cauchy kernel. Next, we decompose each of the sets $\Omega_j$ into two as follows: $E_j:=\{\xi\in\Omega_j: \phi_r(\xi)\geq\phi(\mathbbm{1})/2\}$ and $F_j:=\Omega_j\backslash E_j.$ For each $j$ between $1$ and $m$ the following estimate holds on $E_j$:
\begin{equation*}
|\log\phi_r(\xi)-\log\phi(\mathbbm{1})|\leq\frac{2}{\phi(\mathbbm{1})}|\phi_r(\xi)-\phi(\mathbbm{1})|\leq\varepsilon+\frac{C_0}{\phi(\mathbbm{1})}\left(2^{j}l(Q)\right)^{\alpha}.
\end{equation*}
On the other hand, since $\phi(\mathbbm 1)\leq 1$ we readily get for all $\xi\in F_j$ the following chain of inequalities
\begin{equation*}
|\log\phi_r(\xi)-\log\phi(\mathbbm{1})|=-\log\phi_r(\xi)+\log\phi(\mathbbm{1}) \leq \log\frac{1}{\phi_r(\xi)}.
\end{equation*}
Note that $F_j=\emptyset$ once $j\leq k,$ where $1\leq k\leq m$ is the biggest natural number such that $2^k l(Q)\leq \left(\phi(\mathbbm{1})/2C_0\right)^{1/\alpha}$. Indeed, this can be proved in a way similar to the proof of the inequality~\eqref{strekozzza}. From here we deduce that
\begin{multline}
\label{eq-4}
D\lesssim\frac{\phi(\mathbbm{1})}{|Q|}\int\limits_{Q}\biggl[\sum\limits_{j=1}^m\frac{\left(2^{j}l(Q)\right)^{\alpha}}{\phi(\mathbbm{1})}\int\limits_{E_j}|\widetilde{C}(z,\xi)-\widetilde{C}(\mathbbm{1},\xi)|d\sigma(\xi)+\varepsilon+\biggr.\\
\biggl.\sum\limits_{j=k+1}^m \int\limits_{F_j}\log\frac{1}{\phi_r(\xi)}|\widetilde{C}(z,\xi)-\widetilde{C}(\mathbbm{1},\xi)|d\sigma(\xi) \biggr]d\sigma(z)=: D_1+\varepsilon +D_2.
\end{multline}
We estimate $D_1$ and $D_2$ separately. But before that, we recall one easy lemma whose proof is left to the reader as an exercise. In this lemma we state a usual bound for (the imaginary part of) the Cauchy kernel in the unit ball.
\begin{lem}
\label{yadro}
Let $m$ and $Q$ be as above and let $1\leq j\leq m$. Suppose that $z\in Q$ and $\xi\in \Omega_j$. We have the following inequality
$$|\widetilde{C}(z,\xi)-\widetilde{C}(\mathbbm{1},\xi)|\lesssim \frac{l(Q)}{(2^jl(Q))^{n+1}}.$$
\end{lem}
\noindent
The term $D_1$ is now easy to estimate:
\begin{equation}
\label{eq-3}
D_1\lesssim \frac{\phi(\mathbbm{1})}{|Q|}|Q|\left[\frac{l(Q)^{\alpha}}{\phi(\mathbbm{1})}\sum\limits_{j=1}^m \frac{2^{j\alpha}l(Q)|\Omega_j|}{(2^jl(Q))^{2n+1}}\right] \lesssim l(Q)^{\alpha}.
\end{equation}

We finally proceed to the term $D_2$. First, it follows from the definitions of the functions $f_r$ and from the inequality~\eqref{eq-10} that for all $\xi \in \mathbb S^n, r<1$ and $\rho\in [0,2\pi)$ one has $$\int\limits_{-\rho}^{\rho}|\log|f_r(\xi_1 e^{i\theta},\ldots,\xi_n e^{i\theta})||d\theta\leq B_0.$$
We define sets $Q_j$ as $Q_j=\{z\in \mathbb S^n: d(z,\mathbbm{1})\leq 2^jl(Q)\}$ and choose $\rho:=2^j l(Q)$ for some $j\in \mathbb N$. Integration of the last line with respect to the variable $\xi$ over the set $Q_j$ and further changing variables $\theta$ and $\xi$ gives
\begin{equation}
\label{eq-2}
\rho^{n}\gtrsim 
\int\limits_{-\rho/2}^{\rho/2}\int\limits_{Q_j}|\log|f_r(\xi_1 e^{i\theta},\ldots,\xi_n e^{i\theta})|| d\sigma(\xi) d\theta.
\end{equation}
For each $\theta\in \left(-\rho/2,\rho/2\right)$ and each $\xi\in Q_j$ we define a vector $z$ as $z=(z_1,\ldots,z_n),$ where $z_j:=\xi_j e^{i\theta}.$ We further define a function $F_{\theta}$ by the following formula $F_{\theta}(\xi)=z.$ We finally perform the following change of variables, $z:=F_{\theta}(\xi)$. It follows now from the inequality~\eqref{eq-2} that
\begin{equation}
\label{eq-1}
\rho^{n}\gtrsim \int\limits_{-\rho/2}^{\rho/2}\int\limits_{F_{\theta}(Q_j)}|\log|f_r(z)|||e^{i\theta}|^n d\sigma(\xi) d\theta=\int\limits_{-\rho/2}^{\rho/2}\int\limits_{F_{\theta}(Q_j)}|\log|f_r(z)|| d\sigma(\xi) d\theta.
\end{equation}
We claim that $Q_j/2\subseteq F_{\theta}(Q_j).$ To prove this, we pick a point $\xi\in Q_j/2.$ In order to show that $\xi \in F_{\theta}(Q_j)$ it is sufficient to prove that $\xi e^{-i\theta}\in Q_j$ (for in the last case we can write $\xi=(\xi e^{-i\theta})e^{i\theta}$). We check that $|1-\xi_1 e^{-i\theta}|\leq \rho$ with the help of the triangle inequality:
$$|1-\xi_1 e^{-i\theta}|\leq|1-\xi_1|+|\xi||1-e^{-i\theta}|\leq \frac{\rho}{2}+|\theta|\leq \rho,$$
and our claim follows. The line~\eqref{eq-1} now gives
\begin{equation}
\label{eq0}
\int\limits_{Q_{j-1}}|\log|f_r(z)||d\sigma(z)\lesssim \rho^{n-1}\lesssim (2^{j-1}l(Q))^{n-1}.
\end{equation}
Thanks to the inequality~\eqref{eq0}, we are now ready to finish off the desired bound of the term $D_2$:
\begin{multline}
\label{eq1}
D_2\lesssim \frac{\phi(\mathbbm{1})}{|Q|}\int\limits_{Q}\sum\limits_{j=k+1}^m\int\limits_{Q_j}\log\frac{1}{\phi_r(\xi)}|\widetilde{C}(z,\xi)-\widetilde{C}(\mathbbm{1},\xi)|d\sigma(\xi)d\sigma(z)\lesssim \\ 
\frac{\phi(\mathbbm{1})}{|Q|}|Q|\sum\limits_{j=k+1}^m (2^jl(Q))^{n-1}\frac{l(Q)}{(2^jl(Q))^{n+1}} \lesssim\frac{l(Q)}{\phi(\mathbbm{1})^{\frac{2}{\alpha}-1}}.
\end{multline}
Theorem~\ref{sovsemustal} will now follow from the inequalities~\eqref{eq0},~\eqref{eq-3} and~\eqref{eq-4} simply by letting $\varepsilon$ tend to zero.
\end{proof}

Next, we proceed to the proof of Theorem~\ref{grglthm2}.

\begin{proof}
Let $Q$ be a nonisotropic ball such that $l(Q)^{\gamma}\geq K\phi(\mathbbm{1})$ where $K=(2C_0)^{-\gamma/\alpha}$ for some $0<\gamma\leq\alpha$ to be determined in a moment. Then, from the first claim of Theorem~\ref{sovsemustal} we infer the inequality $\nu(f,Q)\lesssim l(Q)^{\alpha}+l(Q)^{\gamma}.$ On the other hand, if $l(Q)^{\gamma}\leq K\phi(\mathbbm{1})$ for the very same $\gamma$, then the second claim of Theorem~\ref{sovsemustal} provides us with the following estimate
$$\nu(f,Q)\lesssim l(Q)^{\alpha}+l(Q)^{1-\gamma(\frac{2}{\alpha}-1)}.$$
Comparing these inequalities we obtain the following equation:
$\gamma=1-\gamma(2/\alpha-1),$
from where we deduce that $\gamma=\alpha/2.$ In either case, $\nu(f,Q)\lesssim l(Q)^{\alpha/2}$ and Theorem~\ref{grglthm2} follows.
\end{proof}

\section{Proof of Theorem~\ref{grglthm1}}
In the view of the proof of Theorem~\ref{grglthm2}, Theorem~\ref{grglthm1} follows easily and its proof is left to the reader as an exercise. Without giving any details, we only notice that the principal difference between the proofs is that the term $D_2$ in this case can be estimated easier. Indeed, here it suffices to apply the H\"older inequality and utilize the fact that the function $\log \phi$ is in $L^p(\mathbb S^n)$.

\section{Proof of Theorem~\ref{grglthm5}}
Let us now prove that the exponent $p/(p+n)$ is the best possible in Theorem~\ref{grglthm1}.
\begin{theorem}
\label{grglthm5}
Let $p\in\left({n,+\infty}\right).$ Then for each $\delta >0$ there exists an outer function 
$f_0:\mathbb B^n \rightarrow \mathbb C,$ satisfying
$$ f_0 \notin{\mathrm{Lip}_{\frac{\alpha p}{p+n}+\delta}}\left(\mathbbm{1}\right)$$
``in average'' and such that $\log|f_0|\in L^p \left(\mathbb S^n  \right)$ and $|f_0|\in{\mathrm{Lip}_\alpha}\left(\mathbbm{1}\right)$.
\end{theorem}

\begin{proof} 
We precede the proof with one technical lemma.
\begin{lem}
\label{vtorayalemma}
Let $\varepsilon>0,$ and let $\varphi:\mathbb T \rightarrow \mathbb R_+$ be a function such that $\log \varphi \in L^{(p+\varepsilon)/n}(\mathbb T).$
Define a function $f_0:\overline{\mathbb B^n} \rightarrow \mathbb C$ by the formula $ f_0(z_1 ,\ldots, z_n)= g(z_1),$
where $g:\overline{\mathbb D} \rightarrow \mathbb C$ is given by
\begin{equation*}
g(z)=\exp\biggl[\frac{1}{2\pi}\int\limits_0^{2\pi}\frac{e^{i\theta}+z}{e^{i\theta}-z}\log\varphi(e^{i\theta})d\theta\biggl].
\end{equation*}
Then $f_0$ satisfies $\log|f_0| \in{L^p}(\mathbb S^n).$
\end{lem}
\begin{rem}
This statement seems to be folklore, but as we couldn't find a proof in the literature, we present one here.
\end{rem}
\begin{proof}
 Take a point $\zeta=(\zeta_1,\ldots,\zeta_n) \in \overline{\mathbb B^n},$ where $\zeta_1 =re^{i\beta}.$  We write a formula for the function $\log|f|,$ using the definitions of the functions $f_0$ and $g$:
\begin{multline*}
\log |f_0(\zeta_1 ,\ldots,\zeta_n)| = \log|g(\zeta_1)|= \\
 \log \biggl(\exp\biggl[\int \limits_0^{2\pi}\mathrm{Re} \biggl(\frac{e^{i\theta}+\zeta_1}{e^{i\theta}-\zeta_1}\biggr) \log \varphi (e^{i\theta})d\theta \biggr]\biggr) =\\
\int \limits_0^{2\pi}\frac{1-r^2}{1+r^2-2r\cos\left(\beta-\theta \right)} \log \varphi (e^{i\theta}) d\theta=
\log \varphi \ast P_r \left(\beta \right).
\end{multline*}
Hence, thanks to the formula number 1.4.5 from the book of W. Rudin~\cite{rud} (we mean the formula for the integral of a function of fewer variables), we infer the formula
\begin{multline*}
\int \limits_{\mathbb S^n}|\log|f_0\left(\zeta\right)||^p d\sigma\left(\zeta \right)=
\int \limits_0^1 \int \limits_0^{2\pi} r\left(1-r^2 \right)^{n - 2}|\log\varphi\ast P_r (\beta)|^p d\beta dr=\\
\int \limits_0^1 r(1-r^2)^{n-2} \cdot || \log \varphi \ast P_r||_{{L^p}(\mathbb T)}^{p}dr \leq \ldots\,.
\end{multline*}
Next, the Young inequality allows us to continue the estimates:
\begin{equation}
\label{prosto}
\ldots\leq \int \limits_0^1 \left(1-r^2 \right)^{n-2} \cdot || \log \varphi ||_{{L^{\frac{p}{n} + \varepsilon}} \left( \mathbb T \right)}^{p} 
\cdot  ||P_r ||_{L^q \left( \mathbb T\right)}^{p}dr,
\end{equation}
where $q$ stands for the solution of the following equation: 
$$1+\frac{1}{p}=\frac{1}{q}+\frac{1}{\frac{p}{n}+\varepsilon} .$$

It remains to estimate the norm of the Poisson kernel $\|P_r\|_{L^q \left( \mathbb T \right)}$ for $r \in \left(0,1\right)$ and $q\geq 1.$ We first treat the case when $r \in \left(0,1/2\right),$ which turns out to be easy:
\begin{equation*}
\|P_r \|_{L^q \left( \mathbb T \right)}^{q}=\int \limits_0^{2\pi} \frac{\left(1-r^2 \right)^q d\theta }{\left(1+r^2-2r\cos \theta \right)^q} \leq \int \limits_0^{2\pi} \frac{\left(1-r \right)^q d\theta }{\left(1 - r\right)^{2q} } \lesssim \frac {1}{\left(1-r\right)^q} \lesssim \frac{1}{\left(1-r \right)^{q-1}}.
\end{equation*}
Henceforth we assume that $r \in [1/2, 1).$ We are going to use the fact that if
$\theta \in [0,2\pi)$, then $1-\cos \theta \geq C_2\theta^2$ 
with some universal constant $C_2>0$:
\begin{multline*}
\|P_r \|_{L^q \left( \mathbb T \right)}^{q}=\int \limits_0^{2\pi} \frac{\left(1-r^2 \right)^q d\theta }{\left(1+r^2-2r\cos \theta \right)^q}=\\
\int \limits_0^{2\pi} \frac{\left(1-r^2 \right)^q  d\theta } {\left(\left(1-r \right)^2+2r \left(1-\cos \theta \right)\right)^q}\leq 
\int \limits_0^{2\pi} \frac{\left(1-r \right)^q  d\theta} {\left(\left(1-r \right)^2+2r C_2 \theta^2 \right)^q} =\\
\int \limits_A \frac{\left(1-r \right)^q  d\theta} {\left(\left(1-r \right)^2+2r C_2 \theta^2 \right)^q} + \int \limits_B \frac{\left(1-r \right)^q  d\theta} {\left(\left(1-r \right)^2+2r C_2 \theta^2 \right)^q},
\end{multline*}
where $A$ and $B$ stand for the sets
$$A=\{ \theta \in [0,2\pi ):\left(1-r\right)^2 \geq 2C_2r \theta^2 \} $$
and
$$B=\{ \theta \in [0,2\pi ):\left(1-r\right)^2 < 2C_2r \theta^2 \}$$
correspondingly. Note that $A \cap B = \emptyset$ and also that $A\cup B=[0,2\pi). $
Moreover, we trivially have $|A| \lesssim \left(1-r\right).$
Hence, we deduce that
$$ \|P_r \|_{L^q \left( \mathbb T \right)}^{q}\lesssim \int \limits_A \frac{\left(1-r \right)^q  d\theta} {\left(1-r \right)^{2q}} + \int \limits_B \frac{\left(1-r \right)^q  d\theta} { \theta^{2q}} \lesssim \frac{1}{\left(1-r\right)^{q-1}}.$$
We use the inequality that we have just derived along with~\ref{prosto} and write
$$\int \limits_{\mathbb S^n}|\log|f_0\left(\zeta\right)\|^p d\sigma\left(\zeta\right)\lesssim\int\limits_0^1(1-r)^{n-2}\cdot\biggl(\frac{1}{(1-r)^{q-1}}\biggr)^{\frac{p}{q}} dr.$$
It remains to prove that the number $p_1:=(n-2)-p(q-1)/q$ is strictly larger than $-1.$ The definition of the number $q$ yields
\begin{multline*}
p_1=(n-2)-p\frac{(q-1)}{q}=
 \left( n-2\right) -p \biggl( \frac{1}{{\frac{p}{n}} + \varepsilon} - \frac{1}{p} \biggr) = \left( n-2 \right) - \frac{np - p - n\varepsilon}{p+n\varepsilon}=\\
 \frac{n^2 \varepsilon - p - n\varepsilon}{p+n\varepsilon} = \frac{n^2 \varepsilon}{p+n\varepsilon} - 1 > -1,
\end{multline*}
and the lemma follows.
\end{proof}

Denote $p_2=p/n +\varepsilon.$ Since the one-dimensional bound $p/(p+1)$ is the best possible (see~\cite{medved}), there exists a one-dimensional function $\varphi$ satisfying 
$\log \varphi \in L^{p_2} \left( \mathbb T \right)$ and
$ \varphi \in  \mathrm{Lip}_\alpha \left(1\right)$  such that the corresponding outer function
$\mathcal{O}_\varphi$ lies in the ``average space'' $\mathrm{Lip}_{\alpha p_2/(p_2 +1)}\left(1\right)$
 but does not belong to the space $\mathrm{Lip}_{\alpha p_2/({p_2 +1})+\sigma}\left(1\right)$ for each $\sigma >0.$ We construct according to the method described in Lemma~\ref{vtorayalemma} the functions $f_0$ and $g$ (note that the construction there yields $g=\mathcal{O}_\varphi$, recall~\eqref{tratata}). Then it is obvious that $|f_0| \in \mathrm{Lip}_\alpha \left(\mathbbm{1} \right)$, and we deduce from the lemma that $\log |f_0| \in L^p \left( \mathbb S^n \right).$ Hence, according to Theorem~\ref{grglthm1}, 
$f_0 \in \mathrm{Lip}_{\alpha p/(p+n)}\left(\mathbbm{1}\right)$ ``in average''.
It follows from the choice of the function $\varphi $ that 
$f_0 \notin \mathrm{Lip}_{\alpha p_2/(p_2 +1)+\sigma}\left(\mathbbm{1}\right) $
``in average'' for each $\sigma >0.$ On the other hand,
$$\frac{p_2 \alpha}{p_2 +1} + \sigma = \frac{\left( \frac{p}{n} + \varepsilon \right) \alpha}{\left( \frac{p}{n} + \varepsilon \right) +1} + \sigma = \frac{p \alpha + \varepsilon\alpha n}{p+\varepsilon n+n} +\sigma = \frac{\alpha p}{p+n} +\tilde{\varepsilon},$$
where $\tilde{\varepsilon}\rightarrow 0$ once $\varepsilon$ and $\sigma$ tend to zero. Taking $\varepsilon$ and $\sigma$ sufficiently small, we infer that there exists a function $f_0$, satisfying
$f_0 \notin \mathrm{Lip}_{p \alpha/(p+n)+\delta} \left(\mathbbm{1}\right)$
``in average'', where $\delta$ is exactly the same as in the formulation, and the theorem follows.
\end{proof}

\renewcommand{\refname}{References}

\end{document}